\documentclass[11pt]{article}
\newcommand{\sfrac}[2]{{\textstyle\frac{#1}{#2}}}
\usepackage{amsmath,amssymb}
\usepackage{graphicx}  
\newcommand{\Ex}{\mathbb{E}}
 \renewcommand{\Pr}{{\mathbb{P}}}
 \newcommand{\RR}{\mathcal R}
 \newcommand{\YY}{\mathcal Y}
\newcommand{\XX}{\mathcal X}

\newcommand{\TT}{\mathcal T}

\newtheorem{Lemma}{Lemma}

\newtheorem{Proposition}[Lemma]{Proposition}

 \newcommand{\qed}{\ \ \rule{1ex}{1ex}}
 \newcommand{\var}{\mathrm{var}}
 \newcommand{\cov}{\mathrm{cov}}
  \newcommand{\ball}{\mathrm{ball}}
   
\begin{document}

\title{Covering a compact space by fixed-radius or growing random balls}
 \author{David J. Aldous\thanks{Department of Statistics,
 367 Evans Hall \#\  3860,
 U.C. Berkeley CA 94720;  aldous@stat.berkeley.edu;
  www.stat.berkeley.edu/users/aldous.}
}

 \maketitle

 \begin{abstract}
 Simple random coverage models, well studied in Euclidean space, can also be defined on a general compact metric space. 
 By analogy with the geometric models, and with
 the discrete coupon collector's problem and with cover times for finite Markov chains,
 one expects a ``weak concentration" bound for the distribution of the cover time to hold under minimal assumptions.
 We give two such results, one for random  fixed-radius balls and the other for sequentially arriving randomly-centered and
 deterministically growing balls. 
 Each is in fact a simple application of a different more general bound, the former concerning coverage by 
 i.i.d. random sets with arbitrary distribution, and the latter concerning hitting times for Markov chains with a strong
 monotonicity property.
 The growth model seems generally more tractable, and we record some basic results and open problems for that model.
 \end{abstract}

\section{Introduction}
Analogs of the classical coupon collector's problem have been extensively studied in several different contexts.
One context is geometric: covering by (for instance) random balls in Euclidean space \cite{hall_book,penrose2021random}.  
Another context involves the time for an irreducible finite-state Markov
chain to visit every state.
Systematic study of that {\em cover time} $C_{MC}$, particularly for the case of random walks on graphs,
started in the 1980s \cite{me_covering}.
In any context, study of the expectation of the cover time (or more refined study of exact limit rescaled distributions) 
necessarily depends on the specifics of a model, and has been carried out via explicit calculations for many models.
However one expects that the  ``weak concentration" property of the coupon collector time $T_n$
(that s.d.$(T_n) /\Ex T_n \to 0$ as $n \to \infty$)
should extend quite generally to other cover time contexts, and should hold under minimal assumptions even when one cannot calculate the expectation explicitly.  
Indeed this is known to be true in the Markov chain context (see section \ref{sec:discuss}).
The purpose of this article  is to study one analog of geometric covering, in which the Euclidean space is replaced by a metric space.

Another part of our purpose is to spotlight two different general methods (known, but apparently not well known) 
for showing weak concentration in general 
settings without calculating the expectation of the cover time\footnote{Other than its order of magnitude.}.
In each of sections \ref{sec:fixed} and \ref{sec;grow}
we specify a model (fixed-radius or growing random balls), 
recall the relevant general method, and show that
a concentration bound is obtained very easily using that method.  
The growth model seems worthy of further study: we give some more basic results in
section \ref{sec:fa} and pose some challenging open problems.
The special case of the circle is outlined in section \ref{sec:circle}.
Further discussion of  models and methodology is deferred to section \ref{sec:discuss}.

\section{Covering with fixed radius random balls}
\label{sec:fixed}
Here we indicate how a concentration result for
covering, obtainable on Euclidean space in sharp form by explicit calculation \cite{hall_book},
can be extended to weak bounds in a very general setting.
Take a compact metric space $(S,\rho)$.
Let $\mu$ be a probability measure on $S$ with full support, and for $r > 0$ define
\[ \eta(r) := \inf_s \mu(  \ball (s, r)) > 0  \]
where $\ball(s,r) = \{s^\prime: \rho(s,s^\prime) \le r\}$.
Write $\sigma_1, \sigma_2, \ldots$ for i.i.d. random points of $S$ from distribution $\mu$.
For fixed $r_0 > 0$ consider the random subset
\[ \RR_n = \RR_n^{(r_0)} :=  \cup_{1 \le i \le n} \ball (\sigma_i, r_0) .
\]
We call this the {\em fixed-radius model}.
Consider the cover time
\begin{equation}
 C = C^{(r_0)} := \min \{n: \RR_n = S\} 
 \label{C:ball}
 \end{equation}
for which compactness easily implies $\Ex C < \infty$. 
The probability that a given point $s$ is in $\ball(\sigma_i,r_0)$ equals $\mu( \ball (s, r_0))$, and so
the mean time until point $s$ is covered equals $1/\mu(  \ball (s, r_0))$, which is at most
$1/\eta(r_0)$.   So to obtain a concentration result for $C$ a natural  assumption is that $\Ex C \gg 1/\eta(r_0)$,
in other words that $\eta(r_0) \Ex C$ is large.
Our result below is of that general form, but also involves
the {\em dimension}-related quantity $d(r)$ defined as the smallest integer such that 
\begin{equation}
\mbox{ each ball of radius $r$ can be covered by $d(r)$ balls of radius $r/2$. }
\label{dr}
\end{equation}

\begin{Proposition}
\label{P:fixed}
In the fixed-radius model, for the cover time $C$ at (\ref{C:ball}),
\[ \var \left(  \frac{C}{\Ex C} \right) \leq 
\kappa \, \frac{d(r_0)}{\eta(r_0/2)    \Ex C}
\]
for the absolute constant $\kappa$ stated in Proposition \ref{P1} below.
\end{Proposition}
We will derive Proposition \ref{P:fixed} from a known general result, discussed as Proposition \ref{P1} below.

 \subsection{The random subset cover bound}
Here we copy the setup and result directly from  \cite{me_threshold}.
Let $S_0$ be a finite set. Let $\YY$ be a random subset of $S_0$, whose distribution is arbitrary subject to the requirement
\begin{equation}
\mbox{$\Pr(s \in \YY) > 0$ for each $s \in S_0$.}
\label{PY}
\end{equation}
Let $\YY_1, \YY_2, \ldots$ be independent random subsets distributed as $\YY$. Let $\RR_n$ be
the range of this process:
$\RR_n = \cup_{i \le n}   \YY_i$    and let $C_{set}$ be the cover time
\[
C_{set} := \min \{n: \RR_n = S_0\}.
\]
Note $\Ex C_{set} < \infty$ by (\ref{PY}) and finiteness of $S_0$.
For any non-random subset $B \subset S_0$ let $c(B)$ be the mean cover time of $B$:
\[ c(B) := \Ex C(B); \quad C(B):= \min \{n: \RR_n \supseteq B\} . \]
Our bound involves the {\em terminal set }
\[ 
\TT := S_0 \setminus \RR_{C_{set} - 1}
\]
that is the last uncovered portion of $S_0$.

\begin{Proposition}[\cite{me_threshold} Theorem 1]
\label{P1}
$ \var \left(  \frac{C_{set}}{\Ex C_{set}} \right) \leq 
\kappa \, \frac{\Ex c(\TT)}{\Ex C_{set}}$
for an absolute constant $\kappa$.
\end{Proposition}
Though stated in \cite{me_threshold} for a finite state space $S_0$, Proposition \ref{P1} extends to  continuous space, in particular
our compact metric space $S$, with unchanged proof, except that now  we need to 
replace assumption (\ref{PY}) by the assumption $\Ex C_{set} < \infty$.

Of course it may be difficult to analyze $\TT$, and so one does not expect to obtain sharp bounds on specific models in this way.
But Proposition \ref{P1} may be useful in obtaining order of magnitude bounds in general settings.
In particular if there is some geometric or metric structure on the set and if the random subsets $\YY$ are small in diameter, then $\TT$ must be small in diameter, so one needs only to bound $c(B)$ as a function of the diameter of $B$.
The next section gives a simple illustration of that method.

\subsection{Proof of Proposition \ref{P:fixed}}
In the notation of Proposition \ref{P1},  the terminal set $\TT$ is such that $\TT \subset \ball (s, r_0)$ for some $s \in S$, so  
\[ c(\TT) \le \sup_s \Ex C(  \ball (s, r_0)). \]
The mean time until one of the random centers $\sigma$ falls in a given ball of radius $r_0/2$ is at most 
$1/\eta(r_0/2)$.
Note that a ball of radius $r_0/2$ is covered by any ball of radius $r_0$ whose center is in the former ball.
So from the definition of dimension $d$, for each $s$ there are $d$ points $s_1,\ldots, s_d$ such that 
$\ball (s, r_0)$ is covered whenever each of $( \ball (s_i, r_0/2), 1 \le i \le d)$ contains at least one of the random centers $\sigma$, and so
\[ \sup_s \Ex C(  \ball (s, r_0)) \le d/\eta(r_0/2) . \]
The result follows from Proposition \ref{P1}.

\section{The growth model} 
\label{sec;grow}
Consider as before a compact metric space $(S,\rho)$, a probability measure $\mu$ on $S$, but now introduce two rates 
$0 < \lambda < \infty$ and $0 < v < \infty$.
Write 
$0 < \tau_1 < \tau_2 < \ldots$ for the times of a rate-$\lambda$ Poisson process, 
and write $\sigma_1, \sigma_2, \ldots$ for i.i.d. random points of $S$ from distribution $\mu$.
The verbal description
\begin{quote}
seeds arrive at times of a Poisson process at i.i.d. random positions, and then create balls whose
radius grows at rate $v$
\end{quote}
is formalized as the set-valued {\em growth process}
\begin{equation}
\XX(t) := \cup_{i: \tau_i \le t} \ \ball \, (\sigma_i, v(t - \tau_i)).
\label{XXdef}
\end{equation}
We study the cover time
\[
C := \min  \{t: \ \XX(t) = S \}
\]
which is finite because $\Ex \tau_1= 1/\lambda$ and so
\begin{equation}
1/\lambda \le \Ex C \le 1/\lambda + \Delta/v
\label{CD}
\end{equation}
where $\Delta$ is the diameter of $S$. 
To obtain a concentration bound 
it is natural to require that $\Ex C$ is large relative to the maximum expected time to cover any 
given single point, that is relative to 
\[
c^* := \max_{s \in S} \Ex C(s); 
\quad
C(s) := \min  \{t: \ s \in \XX(t)\} .
\]
It turns out this is the only requirement.
\begin{Proposition}
\label{P:growth}
In the growth model (\ref{XXdef}),
$\var \left( \frac{C}{\Ex C} \right) \le \frac{c^*}{\Ex C}$.
\end{Proposition}
We will derive Proposition \ref{P:growth} from a known general result, discussed as Proposition \ref{L-simple-1} below.
Note that the expectation of the number of balls covering $v$ at time $t$ equals
$\int_0^t \mu(\ball(s,vu)) \   \lambda du    $
and so from the Poisson property
\begin{equation}
\Pr(C(s) > t) = \exp \left( -    \int_0^t \mu(\ball(s,vu)) \   \lambda du        \right) 
\label{PCs}
\end{equation}
from which we can in principle obtain a formula for $\Ex C(s)$.

\subsection{A monotonicity bound for Markov chains}
\label{sec:monot}
Here we copy the setup and result directly from  \cite{me-weakC}.
The setting there is a continuous-time Markov chain $(X_t)$ on a finite state space $\Sigma$, 
where we study the hitting time 
\begin{equation}
 T := \inf \{t: \ X_t \in \Sigma_0\} 
 \label{T-def}
 \end{equation}
for a fixed subset $\Sigma_0 \subset \Sigma$.  
Assume
\begin{equation}
\label{h-def}
 h(x) := \Ex_x T < \infty \mbox{ for each } x \in \Sigma 
 \end{equation}
which holds in the finite case under the natural ``reachability" condition.
Assume also a rather strong ``monotonicity" condition:
\begin{equation}
h(x^\prime) \le h(x) \mbox{ whenever $x \to x^\prime$ is a possible transition}. 
\label{def-monotone}
\end{equation}
\begin{Proposition}[\cite{me-weakC}]
\label{L-simple-1}
Under conditions (\ref{h-def}, \ref{def-monotone}), for any initial state,
\[ \frac{\var  \ T}{\Ex  T} \le  \max \{ h(x) - h(x^\prime)    : \ x \to x^\prime \mbox{ a possible transition} \} .
\]
\end{Proposition}
Though stated in  \cite{me-weakC} for a finite state space $\Sigma$, Proposition \ref{L-simple-1} extends to  continuous space with essentially unchanged proof.

\subsection{Proof of Proposition \ref{P:growth}}
\label{sec:PPg}
The cover time $C$ for our growth model $\XX(t)$ at (\ref{XXdef}) is of the form in Proposition \ref{L-simple-1}; 
the state space is the space of compact subsets $x$ of the compact metric space $S$.
The only discontinuities of $h(\XX(t))$ are at a time $\tau$ when a new seed arrives at a point $\sigma$,
at which time there is a transition $x \to x \cup \{\sigma\}$ of $\XX(t)$.
To apply Proposition \ref{L-simple-1} to prove Proposition \ref{P:growth} it is enough to show 
that, for each pair $(x,\sigma)$,
\begin{equation}
h(x) - h(x \cup \{\sigma\}) \le \Ex C(\sigma) .
\label{hhx}
\end{equation}
But this holds by considering the natural coupling $(\XX(t), \XX^\prime(t) = \XX(t) \cup \ball(\sigma,vt)  , t \ge 0)$ 
of the growth processes with $\XX(0) = x, \XX^\prime(0) = x \cup \{ \sigma\}$.
In this coupling, for the time $C^*(\sigma)$ at which $\sigma$ is reached by a ball of $\XX(\cdot)$ 
whose seed arrived after time $0$, we have 
(by the triangle inequality on $S$)
that $\XX( C^*(\sigma) + t) \supseteq \XX^\prime (t)$,
and so the cover times for these two processes differ by at most $C^*(\sigma)$. 
But this  $C^*(\sigma)$ is distributed as $C(\sigma)$ for the growth process started at the empty set,
establishing (\ref{hhx}).

\section{Further analysis of the general growth model}
\label{sec:fa}
Comparing the statements of Propositions \ref{P:fixed} and \ref{P:growth} suggests that the growth model
is more tractable for the study of covering.
Intuitively this is because the behavior of the growth model is ``smoother" in that it does not rely on the detailed geometry of the space $(S,\rho)$ at the given distance $r_0$.
In this section we record some simple observations and then pose some open problems.

We can ``standardize" the growth model by choosing time and distance units to make $\lambda = v = 1$.
With this standardization we have a relationship between the diameter $\Delta$ and $\Ex C$.
\begin{Proposition}
\label{ECD}
In the standardized growth model on a space $(S,\rho)$,\\
(a) $\Ex C \le 1 + \Delta$.\\
(b) If $S$ is connected then $\Delta \le \kappa_1 (\Ex C)^2$ for an absolute constant $\kappa_1$.
\end{Proposition}
\begin{proof}
Part (a) is (\ref{CD}).  
For (b), at time $t$ the sum of diameters of balls is at most
\[ D(t) := 2 \sum_i (t - \tau_i)^+ .\]
By connectedness we must have
\[ \Delta \le D(C) . \]
We can rewrite $D(t)$ in terms of the Poisson counting process $(N(t))$ as 
$D(t) = 2 \int_0^t N(u) du$ and then
\[ \Delta \le \Ex D(C) = 2 \int_0^\infty \Ex [N(t) 1_{(t \le C)} ] \ dt . \]
Using the Cauchy-Schwarz inequality
\begin{equation}
 \Delta \le 2 \int_0^\infty (t^2 + t)^{1/2} \ \sqrt{\Pr(C \ge t)} \ dt .
\label{Dint}
\end{equation}
Now the obvious submultiplicative property of the cover time $C$, that is
\[ \Pr(C \ge t_1 + t_2) \le \Pr(C \ge t_1) \ \Pr(C \ge t_2) \]
combined with Markov's inequality $\Pr(C \ge e \Ex C) \le e^{-1}$
implies an exponential tail bound
\begin{equation}
 \Pr ( C \ge t) \le \exp( 1 - \sfrac{t}{e \Ex C} ) 
 \label{submult}
 \end{equation}
and the result follows from (\ref{Dint}) and straightforward calculus bounds 
(note $\Ex C \ge 1$ from (\ref{CD})).
\qed
\end{proof}

Continuing with this standardization, consider a sequence of connected compact metric spaces $S = S^{(n)}$ and probability distributions $\mu = \mu^{(n)}$.
Proposition  \ref{P:growth} implies that as $n \to \infty$
\begin{equation}
\mbox{ if } \sfrac{c^*}{\Ex C} \to 0 \mbox{ then } \sfrac{C}{\Ex C} \to 1 \mbox{ in } L^2 .
\label{cC}
\end{equation}
Can we relate the  hypothesis  $c^*/\Ex C \to 0$ to other aspects of the spaces?
Recall that $c^*$ is in principle directly calculable from (\ref{PCs}), whereas
determining whether $\Ex C$ is of the same order, or larger order, than $c^*$ requires some more 
detailed knowledge of the space $S$.

If the diameters $\Delta^{(n)}$ are bounded (as $n$ increases) then by Proposition \ref{ECD} the mean cover times 
$\Ex C^{(n)}$ are bounded; because $\Pr(C^{(n)} > t) \ge \exp(-t)$ the conclusion (and hence the assumption) 
of (\ref{cC}) is false.  
So we need study only the case $\Delta^{(n)} \to \infty$.
Here is a simple example to show that the conclusion of (\ref{cC}) is not always true.

{\bf Example.}
Take $S^{(n)}$ to be the real line segment $[0,n]$ and $\mu^{(n)}(0) = 1 - 1/n$ and $\mu^{(n)}(n) = 1/n$.
One easily sees that 
\[ n^{-1} C^{(n)} \to_d \min(1, \sfrac{1}{2}(1+\xi) ) \]
where $\xi$ has Exponential(1) distribution.

In an opposite direction, we note a simple upper bound on $\Ex C/c^*$, that is a lower bound on $c^*/\Ex C$,
in terms of the {\em covering number}
\begin{equation}
 \cov(r) := \mbox{ minimum number of radius $r$ balls that cover $S$ }. 
 \label{def:cov}
 \end{equation}
\begin{Proposition}
In the standardized growth model,
\[ \frac{\Ex C}{c^*} \le \min_{a > 0}  [a + e(e + \log \cov(ac^*) ) ]        .\]
\end{Proposition}
\begin{proof}
As at (\ref{submult}) the submultiplicative property of $C(s)$ implies
$\Pr ( C(s)  \ge t) \le \exp( 1 - \sfrac{t}{e \Ex C(s)})$.
Applying this to the centers $(s_i)$ of $\cov(r)$ covering radius $r$ balls,
\[ \Pr ( \max_i C(s_i) \ge t) \le  e\ \cov(r)  \exp( - \sfrac{t}{e c^*} ). \]
Setting $t_0 := ec^* \log \cov(r)$,
\[ \Ex [\max_i C(s_i) ] = \int_0^\infty  \Pr ( \max_i C(s_i) \ge t) \; dt
\le t_0 + e \cdot ec^*. \]
 Because $C \le r + \max_i C(s_i)$ we have
 \[ 
 \Ex C \le r + ec^* (e + \log \cov(r)) . \]
 Setting $r = ac^*$ gives the stated bound.
\qed
\end{proof}

\subsection{The minimizing seed distribution}
 For the standardized growth model on connected compact $(S,\rho)$, take two points $s_1,s_2$ which are diametrically opposite, that is 
$\rho(s_1,s_2) = \Delta$. 
Then the maximum of $\Ex_{\mu} C$ over $\mu$ equals $1 + \Delta$,  attained by the measure $\mu$ degenerate at $s_1$.
But what can we say about  the {\em minimum} of $\Ex_{\mu} C$ over $\mu$?

Intuitively this should be related to the covering numbers $\cov(r)$ at (\ref{def:cov}).
And indeed there is a simple upper bound in terms of the covering numbers.  
Given $r$, consider $\mu$ uniform on the centers $(s_i, 1 \le i \le \cov(r))$ 
of the covering radius-$r$ balls.
Then $ C \le r + \tau_{\cov(r)}$ where $\tau_n$ is the 
elementary coupon collector time
with $\Ex \tau_n = n(1 + 1/2 + \ldots + 1/n) \le (1 + \log n) n$.
So we have established
\begin{Proposition}
In the standardized growth model, 
\[ \min_\mu \Ex_{\mu} C \le \min_{r>0} [r + \cov(r) (1 + \log \cov(r))] . \]
\label{P:sg1}
\end{Proposition}
For a bound in the opposite direction, observe first that for the Poisson counting process $(N(t), 0 \le t < \infty)$ of seed arrival times,
\begin{Lemma}
\label{L:NP}
If $t_0$ and $c_0$ are such that 
$\Pr(C > c_0) + \Pr(N(c_0) > t_0) < 1$
then
$\cov(c_0) \le t_0$.
\end{Lemma}
\begin{proof}
The assumption implies that the event $\{C \le c_0, N(c_0) \le t_0 \}$ has non-zero probability; 
on that event we have
\[ \cov(c_0) \le N(C) \le N(c_0) \le t_0 . \]
\qed
\end{proof}

Applying Lemma \ref{L:NP} with $c_0 = 3 \Ex C$ and $t_0 = 3c_0$ gives 
$\cov( 3 \Ex C) \le 9 \Ex C$.
This is true for any $\mu$ and so 
\begin{Proposition}
In the standardized growth model, 
\[ \min_\mu \Ex_{\mu} C \ge \min\{r : \cov(3r) \le 9r\}   . \]
\label{P:sg2}
\end{Proposition}
Roughly speaking, if the space is $d$-dimensional in the sense that 
$\cov(r) \asymp (A/r)^d$ for fixed large $A$, then Propositions \ref{P:sg1} and \ref{P:sg2} imply that
$\min_\mu \Ex_{\mu} C $ is between orders $A^{\frac{d}{d+1}}$ and $A^{\frac{d}{d+1}} \log A$.

\subsection{Open problems for  the general growth model}
\label{sec:OPs}
\begin{itemize}
\item As mentioned above, can we find easily checkable conditions to ensure that $c^*/\Ex C \to 0$?
\item Can one improve the upper and lower bounds on $\min_\mu \Ex_\mu C$ above? 
In particular, can $\min_\mu \Ex_\mu C$ be more sharply related to some measure of {\em entropy}
of the metric space (see e.g. \cite{entropy}
 for possible notions of entropy)?
\item For $\mu$ attaining the minimum $\min_\mu \Ex_\mu C$, do we always have weak concentration?
That is, is there 
a function $\psi(\Delta) \downarrow 0$ as $\Delta \uparrow \infty$ such that on every connected compact 
metric space, for the standardized growth model,
\[ \var_\mu \left( \frac{C}{\Ex_\mu C} \right) \le \psi(\Delta)\]
 for the minimizing $\mu$?
 \item Is there an effective algorithmic procedure for finding a minimizing $\mu$? 
 This seems  loosely similar to the well-studied {\em k-median problem} \cite{k-median}.

 \item If $S$ is a compact group, with a metric invariant under the group action, then is the uniform (Haar) measure the minimizing measure?
\end{itemize}
Regarding the final problem above, it can be shown that, on the circle of integer circumference $L$, for the fixed-radius model with $r = 1/2$, 
the mean cover time for seed distribution $\mu$ uniform on $L$ evenly-spaced points is smaller than
that for $\mu$ uniform on the circle (the discrete analog is noted in  \cite{coupons_2020} Example 4.1).
We do not know if this type of example is a counter-example in the growth model; 
if so, replace by an asymptotic ($\Delta \to \infty$) conjecture.

\section{The growth process on the circle}
\label{sec:circle}
Here we consider the standardized growth model on the circle $S$ of circumference $L$, with uniform distribution $\mu$.
What is the $L \to \infty$ limit distribution of the cover time $C(L)$?
We will treat this as another example where the  {\em Poisson clumping heuristic} (PCH) \cite{me-PCH} gives a recipe for 
calculating explicitly the limit distribution; the method is {\em heuristic} in the sense of not justifying the approximations, 
but would provide a template for making a rigorous proof.

\subsection{The calculation}
Consider an interval $A \subset S$ of length $a \ll L$.
The number $N_A(t)$ of balls intersecting $A$ at time $t$ has Poisson distribution with
\[ \Ex N_A(t) = \int_0^t \min(a+2u,L) \ L^{-1} du . \]
We will use this only when $a$ and $t$ are order $L^{1/2 + o(1)}$ and so we can ignore the truncation.
This gives (the equalities below are really approximations)
\[ \Ex N_A(t) = (at + t^2)/L \]
\begin{equation}
 \Pr( N_A(t) = 0) = \exp(- (at + t^2)/L ) . 
 \label{PNA}
 \end{equation}
The probability $p(t)$ that a specified point $s \in S$ is not covered at time $t$ 
is the case $a = 0$, so
\begin{equation}
 p(t) = \exp( -t^2/L) .
 \label{circlept}
 \end{equation}
As $t$ approaches $C(L)$ the uncovered region is a union of intervals of lengths small relative to $L$.
The PCH asserts that,  as a good approximation which gives correct asymptotics,  
one can assume these intervals have i.i.d. lengths (with some distribution $\Lambda(t)$) and their centers
are as a Poisson process (of some rate $\lambda(t)$ per unit length); 
these quantities are related by
\[ p(t) = \lambda(t) \Ex \Lambda(t) . \]
The number of uncovered intervals therefore has Poisson distribution with mean $L \lambda(t)$
and so
\[ \Pr(C(L) \le t) = \exp( - L \lambda(t)) . \]
In this example it is easy to ascertain the distribution of $\Lambda(t)$.
Given that a point $s$ is uncovered, the conditional probability that the interval 
$[s - a_1,s+a_2]$ is uncovered (that is, is a subset of the uncovered interval containing $s$) equals,
by (\ref{PNA}),
$\exp( - (a_1+ a_2)t/L)$,
and therefore the whole uncovered interval $[s-A_1,s+A_2]$ is such that 
$A_1$ and $A_2$ are independent with Exponential($t/L$) distribution.
This length $A_1 + A_2$ is the size-biased distribution of $\Lambda(t)$, so its un-size-biased distribution is 
just Exponential($t/L$), with expectation
\[ \Ex \Lambda(t) = L/t . \]
Combining the displayed equations above gives us
\begin{equation}
\Pr(C(L) \le t) \approx  \exp( - t e^{-t^2/L}) 
\label{PCL}
\end{equation}
where we are now acknowledging that this is 
an approximation\footnote{For large $L$, and $t$ not in the tails.}, expected to lead to the correct asymptotics.

\subsection{Asymptotics}
And indeed (\ref{PCL})  corresponds, as one expects from general extreme value theory \cite{resnick}, to a limit result 
of the form
\begin{equation}
(C(L) - t_0(L))/\sigma(L) \to_d \zeta, \quad \Pr(\zeta \le x) = \exp(-e^{-x}), - \infty < x < \infty . 
\label{circle_limit}
\end{equation}
To make this explicit, define $G(y)$ to be the inverse function of $y = x \exp(-x^2)$ for large $x$ and small $y$
and then define
\begin{equation}
 t_0(L) := L^{1/2} G(L^{-1/2}) 
 \label{def:t0}
 \end{equation}
so that (\ref{PCL})  becomes 
\[ 
\Pr(C(L) \le t_0(L) ) \approx \exp( - 1) . 
\]
One  can now calculate from (\ref{PCL}) that for fixed $x$
\[
\Pr\left(C(L) \le t_0 + \sfrac{xL^{1/2}}{2G(L^{-1/2})}  \right) \approx \exp(-e^{-x}) 
\]
corresponding to (\ref{circle_limit})  with $t_0(L)$ defined by (\ref{def:t0}) and $\sigma(L)$ defined by
\[
\sigma(L) := \sfrac{L^{1/2}}{2G(L^{-1/2})}  .
\]
The function $G(L^{-1/2})$ is slowly varying, roughly as $\sqrt{\log L}$.

For comparison with the general result of Proposition \ref{P:growth},
note that from (\ref{circlept})
\[ c^* = \Ex C(s) =  \int_0^\infty \exp(-t^2/L) \ dt = \sfrac{1}{2} \pi^{1/2} L^{1/2} \]
and so 
\[ 
\var \left( \frac{C}{\Ex C} \right) \asymp \frac{1}{G^{4}(L^{-1/2})}, \quad
\frac{c^*}{\Ex C} \asymp \frac{1}{G(L^{-1/2})}. 
\]
The right side is the upper bound (Proposition \ref{P:growth}) and the left side is the correct order of magnitude.

\section{Discussion}
\label{sec:discuss}

\subsection{Comments on the two models} 
The fixed-radius model is a natural generalization of covering Euclidean space
with balls, though this generalization to metric spaces has apparently has not been studied before.
The growth model in our simple form has also apparently not been studied, though it 
can be regarded as an extremely basic 
model for the spread of information or the spread of an epidemic, 
a field with a huge literature studying models on graphs or Euclidean space
\cite{epi1,epi2,epi3}. 
A related growth model in two dimensions, where seeds arrive (instead of as a constant-rate process) as a Poisson process whose rate is the current occupied area, is studied in
\cite{me-gossip,durrett-gossip}.

\subsection{The growth model on other spaces}
In addition to the open problems in section \ref{sec:OPs},
there is much scope for further study of the growth model on specific spaces.
As well as other classical compact spaces familiar from analysis, 
one can consider a finite graph with edge lengths, with the metric of shortest route length.
Moreover there are {\em random} metric spaces of contemporary interest in probability, such as
the ``mean-field model of distance" \cite{PWIT},  
the Brownian CRT \cite{brownian_crt},
or the  Brownian map \cite{brownian_map}.

\subsection{Other uses of the two general bounds}
We have used  two general methods 
-- the random subset cover bound (Proposition \ref{P1})
and the monotonicity bound (Proposition \ref{L-simple-1}) -- 
which are in principle applicable in very general covering-like contexts 
to establish weak concentration bounds in general 
settings without calculating the expectation of the covering time.  
We provide some history of these methods below, and speculate that there may be other applications not yet explored.

{\bf The random subset cover bound},  Proposition \ref{P1}, for general i.i.d. random subsets of a set, 
was given in  \cite{me_threshold} as part of the proof of a weak concentration bound for
the Markov chain cover time  $C_{MC} $.
In the Markov chain context,  the i.i.d. subsets arise as excursions from a given state.
 In the result, the essential condition is that the maximum mean hitting time to any single state is $o(\Ex C_{MC} )$.
 In that sense the bound is closely analogous to the bounds in this article.
In the 30 years since \cite{me_threshold}, study of random walk cover times has entered a more 
sophisticated phase based on the discovery \cite{ding_peres}
of its connection with Gaussian free fields and Talagrand's theory of majorizing measures. 
In contrast, the program of using general results for i.i.d. random subsets as part of analysis 
of specific contexts within covering
seems not to have been developed until the recent work \cite{coupons_2020}.
That paper discusses known results in  combinatorial settings,
develops new general results and applies them to
several topics:
connectivity in random graphs;
covering a square with random discs;
covering the edges of a graph by spanning trees, and matroids by bases;
and random $k$-SAT. 

{\bf  The monotonicity bound}, Proposition \ref{L-simple-1}, was given in \cite{me-weakC} as a tool for establishing
weak concentration for first passage percolation times on general graphs.  
It was also used \cite{me-incipient} for weak concentration of the time of emergence of the 
giant component in bond percolation on general graphs.
Both contexts involve hitting time of an increasing set-valued Markov process, as does our application in section \ref{sec:PPg}.

 
 \end{document}